\documentclass[12pt]{amsart}

\usepackage{amsfonts,amsthm,amsmath,amssymb,amscd,mathrsfs,tikz}
\usepackage{graphics}
\usepackage{indentfirst}
\usepackage{cite}
\usepackage{latexsym}
\usepackage[dvips]{epsfig}
\usepackage{color}
\usepackage{enumerate}
\usepackage{verbatim}
\usetikzlibrary{patterns}

\usepackage{hyperref}

\newtheorem{theorem}{Theorem}[section]
\newtheorem{remark}{Remark}[section]

\newtheorem{lemma}{Lemma}[section]

\newtheorem{corollary}[lemma]{Corollary}

\renewcommand{\div}{{\rm div\, }}
\newcommand{\curl}{{\operatorname{curl}\, }}
\newcommand{\p}{\partial}
\newcommand{\R}{\mathbb{R}}
\renewcommand{\r}{\rho}
\newcommand{\te}{\theta}

\newcommand{\vp}{\varphi}

\newcommand{\nb}{\nabla}
\newcommand{\B}{\mathbf{B}}
\newcommand{\bu}{\mathbf{u}}
\newcommand{\bn}{\mathbf{n}}
\def\e{\mathbf{e}}

\def\bw{\mathbf{w}}
\def\bU{\mathbf{U}}
\def\bt{\boldsymbol{\tau}}
\def\bb{\mathbf {b}}

\begin{document}

\title[On axisymmetric self-similar solutions to the MHD system]{On axisymmetric self-similar solutions to the MHD system}

\author{Shaoheng Zhang} 
\address{School of Mathematical Sciences, Soochow University, Suzhou, 215006, China}
\email{20234007008@stu.suda.edu.cn}


\begin{abstract}
Let $(\mathbf{u},\mathbf{B})$ be an axisymmetric self-similar solution to the stationary MHD equations with magnetic diffusion, of the form  
$\mathbf{u}=u^r(r,z)\mathbf{e}_{r}+u^{\theta}(r,z)\mathbf{e}_{\theta}+u^z(r,z)\mathbf{e}_{z}$ and $\mathbf{B}=B^{\theta}(r,z)\mathbf{e}_{\theta}$ in cylindrical coordinates $(r,\theta,z)$, where $(\mathbf{e}_r,\mathbf{e}_\theta,\mathbf{e}_z)$ is the orthonormal basis.  
Under the assumption that $u^r < \frac{1}{3r} + \frac{2r}{3}$ on the unit sphere and on its intersection with the half-space, respectively, we prove two main results. First, for the domain $\mathbb{R}^3\setminus\{0\}$, the velocity field $\mathbf{u}$ must be a Landau solution and the magnetic field $\mathbf{B} \equiv 0$. Second, in the half-space $\mathbb{R}^3_+$ with either the no-slip or Navier slip boundary condition, we establish that all such axisymmetric self-similar solutions are trivial, i.\,e., $\mathbf{u}=\mathbf{B}=0$.


\end{abstract}

\maketitle

\section{Introduction}

We investigate the stationary incompressible magnetohydrodynamic (MHD) system with magnetic diffusion in a domain $\Omega \subseteq \mathbb{R}^3$:
\begin{align}\label{MHD}
\begin{cases}
-\Delta \mathbf{u} + (\mathbf{u} \cdot \nabla)\mathbf{u} - (\mathbf{B} \cdot \nabla)\mathbf{B} + \nabla p = 0, \\
-\Delta \mathbf{B} + (\mathbf{u} \cdot \nabla) \mathbf{B} - (\mathbf{B} \cdot \nabla) \mathbf{u} = 0, \\
\nabla \cdot \mathbf{u} = \nabla \cdot \mathbf{B} = 0,
\end{cases}
\end{align}
where $\mathbf{u}$, $\mathbf{B}$, and $p$ represent the fluid velocity, magnetic field, and scalar pressure, respectively. This system models the dynamics of electrically conducting fluids (e. g., plasmas) by coupling the Navier--Stokes equations for fluid flow with Maxwell's equations for electromagnetism.

The MHD system \eqref{MHD} is invariant under the following scaling:
\begin{align*}
    \bu(x)\to \bu_{\lambda}(x)=\lambda \bu(\lambda x),\ \
    \B(x) \to \B_{\lambda}(x)=\lambda \B(\lambda x),\ \
    p(x)\to p_{\lambda}(x)=\lambda^2 p(\lambda x),
\end{align*}
where $\lambda>0$.
The solution $(\bu,\B,p)$ is called {\it self-similar} if 
\[
\bu(x)=\bu_{\lambda}(x), \quad \B(x)=\B_{\lambda}(x),\quad p(x)=p_{\lambda}(x)
\] 
for any $\lambda>0$. The goal of this paper is to characterize these self-similar solutions in the axisymmetric case (see Section~\ref{Sect2.1} for the definition of axisymmetric vector fields).

A fundamental special case occurs when $\mathbf{B} \equiv 0$, in which the MHD system \eqref{MHD} reduces to the Navier--Stokes equations.
The well-known self-similar solutions to the Navier--Stokes equations in $\R^3\setminus\{0\}$ are the {\it Landau solutions} (see \cite{Landau44}), which are axisymmetric with no swirl and have exactly one point singularity at the origin.
Recall that a function $f$ is called $(-1)$-homogeneous if $f(\lambda x)=\lambda^{-1}f(x)$ for all $\lambda>0$, and $(-2)$-homogeneous if $f(\lambda x)=\lambda^{-2}f(x)$. The Landau solutions
can be parameterized in the following way: for any $0 \neq \bb \in \R^3$, there exists a unique $(-1)$-homogeneous $\bU^\bb$ and a corresponding $(-2)$-homogeneous $P^\bb$ such that $(\bU^\bb,P^\bb)$ are smooth in $\R^3\setminus \{0\}$ and they solve
\[
-\Delta \bU^\bb + \bU^\bb \cdot \nabla \bU^\bb + \nb P^\bb =\bb\delta, 
\quad \nabla \cdot \bU^\bb = 0
\]
in the distributional sense on $\R^3$, where $\delta$ denotes the Dirac delta function.
When $\bb=(0,0,\beta)$ with $\beta>0$, $(\bU^\bb,P^\bb)$ has the explicit form
\begin{align*}
    \bU^\bb=\frac{2}{\rho} 
    \big(\frac{a^2-1}{(a-\cos \vp)^2}-1\big)\e_{\rho}
    -\frac{2\sin \vp}{\r(a-\cos \vp)}\e_{\vp},
    \quad
    P^\bb=\frac{4(a \cos \vp-1)}{\rho^2 (a-\cos \vp)^2}
\end{align*}
in spherical coordinates $(\r,\te,\vp)$ with $(x_1, x_2, x_3) =(\r\sin\vp\cos \te,\r\sin\vp\sin \te, \r \cos \vp)$,
and $\mathbf{e}_{\rho}=\frac{x}{\rho},\mathbf{e}_{\theta}=(-\sin \theta, \cos \theta,0), \mathbf{e}_{\varphi}=\mathbf{e}_{\theta}\times \mathbf{e}_{\rho}$.
The parameters $\beta>0$ and $a > 1$ are related by $\beta=16 \pi 
\big(a + \frac 12 a^2 \log \left(\frac {a-1}{a+1}\right) + \frac{4a}{3(a^2-1)}\big)$ (see also \cite[Section 8.2]{Tsai18}).
The characterization of self-similar solutions in $\R^3\setminus\{0\}$ as Landau solutions under axisymmetry was established by Tian and Xin \cite{TX98} and, separately, by Cannone and Karch \cite{CK04}.
\v{S}ver\'ak established the following result in \cite{Sverak11}:
\begin{theorem}
All $(-1)$-homogeneous nontrivial solutions to the Navier--Stokes equations in $C^{\infty}(\R^3 \setminus \{0\})$ are Landau solutions.
\end{theorem} 
It was shown by Kang, Miura, and Tsai \cite{KMT18} that all axisymmetric self-similar solutions to the Navier--Stokes equations in $\R^3_+$ are trivial, subject to either the no-slip or Navier slip boundary condition.

Regarding the self-similar solutions to the stationary MHD equations, fewer results are known. 
The structure of axisymmetric solutions to the stationary MHD equations with $\mathbf{B} = B^{\theta} \mathbf{e}_{\theta}$ has been characterized under different assumptions. In the case with magnetic diffusion, Zhang, Wang, and Wang \cite{ZWW25} proved that $\mathbf{u}$ must be a Landau solution and $\mathbf{B}$ vanishes, given that $|\mathbf{u}(x)|$ is sufficiently small on the unit sphere. For the case without magnetic diffusion, Zhang \cite{Zhang25} proved that the same conclusion holds unconditionally. 


This paper is concerned with axisymmetric self-similar solutions to the MHD system with magnetic diffusion ---namely, \eqref{MHD}---and our first result states
\begin{theorem}\label{thm1}
Suppose $(\bu,\B)$ is a smooth axisymmetric self-similar solution to \eqref{MHD} in $\R^3\setminus\{0\}$ of the form $\bu=u^{r}(r,z)\e_r+u^{\te}(r,z)\e_{\te}+u^{z}(r,z)\e_z$ and $\B=B^{\te}(r,z)\e_\te$, where $(r,\te,z)$ are the cylindrical coordinates defined in Section \ref{Sect2.1}. 
If the inequality $u^{r}(r,z) < \frac{1}{3r}+\frac{2r}{3}$ holds on the unit sphere $\partial B_1$, then $\mathbf{u}$ must be a Landau solution and $\mathbf{B} \equiv 0$.
\end{theorem}

As an immediate consequence of Theorem \ref{thm1}, we obtain the following corollary.

\begin{corollary}\label{cor1}
Suppose $(\bu,\B)$ is a smooth axisymmetric self-similar solution to \eqref{MHD} in $\R^3\setminus\{0\}$ of the form $\bu=u^{r}(r,z)\e_r+u^{\te}(r,z)\e_{\te}+u^{z}(r,z)\e_z$ and $\B=B^{\te}(r,z)\e_\te$, where $(r,\te,z)$ are the cylindrical coordinates defined in Section \ref{Sect2.1}. 
If $u^r(x)\le \alpha$ on $\p B_1$ with $\alpha<\frac{2\sqrt 2}{3}$, then $\bu$ is a Landau solution and $\B\equiv0$.
\end{corollary}

\begin{remark}
Let $(\mathbf{u},\mathbf{B})$ be a smooth axisymmetric self-similar solution to \eqref{MHD} in $\mathbb{R}^{3}\setminus \{\mathbf 0\}$ with $\mathbf{B} = B^{\theta}(r,z)\mathbf{e}_{\theta}$. 
Theorem \ref{thm1} provides a general condition, $u^{r}< \frac{1}{3r} +\frac{2r}{3}$ on $\partial B_{1}$, under which $\mathbf{u}$ is a Landau solution and $\mathbf{B}\equiv 0$. A special case of this condition is $u^{r}(x)\leq \alpha$ on $\partial B_{1}$ with $\alpha < \frac{2\sqrt{2}}{3}$, as stated in Corollary \ref{cor1}. This corollary is particularly noteworthy because it gives an explicit numerical bound that improves upon a previous result in \cite{ZWW25}, where the same conclusion was obtained under the assumption that $|\mathbf{u}(x)|\leq \varepsilon$ on $\partial B_{1}$ with a sufficiently small $\varepsilon>0$. Thus, Corollary \ref{cor1} demonstrates that controlling only the $u^r$ component with an explicit constant is sufficient, which is a weaker requirement than controlling the full magnitude $|\mathbf{u}|$.
\end{remark}

Our second main result establishes that all axisymmetric self-similar solutions to the MHD system \eqref{MHD} in the half-space $\mathbb{R}^3_+$ are trivial, under either the no-slip or the Navier slip boundary condition.

\begin{theorem}\label{thm2}
Suppose $(\mathbf{u},\mathbf{B})$ is a smooth axisymmetric self-similar solution to \eqref{MHD} in $\R^3_+$ of the form $\bu=u^{r}(r,z)\e_r+u^{\te}(r,z)\e_{\te}+u^{z}(r,z)\e_z$ and $\B=B^{\te}(r,z)\e_\te$, satisfying either the no-slip boundary condition 
\begin{align}\label{noslipbc}
    \bu=0,  \quad \B\cdot \bn= 0, \quad \curl \B \times \bn=0, \quad  &\text{on  }\ \p\R^3_+\setminus\{0\},
\end{align}
or the Navier slip boundary condition \begin{equation}\label{navierslipbc}
\begin{aligned}
    \bu\cdot \bn=0, \quad (\mathbb{D}\bu \cdot \bn)\cdot \bt=0  , \quad\B\cdot \bn= 0, \quad  \curl \B \times \bn=0, \quad  &\text{on  }\ \p\R^3_+\setminus\{0\}.
\end{aligned}
\end{equation} 
Here $\bn$ and $\bt$ are the unit outward normal and tangent vectors on $\p\R^3_+$, respectively, and $\mathbb{D}\bu=\frac{1}{2}(\nb \bu+\nb \bu^T)$  is the  strain tensor.
Then, under the condition that $u^{r}(r,z)< \frac{1}{3r}+\frac{2r}{3}$ on $\p B_1 \cap \R^3_+$, we have $\mathbf{u}=\mathbf{B}= 0$.

\end{theorem}

\begin{remark}
As in Corollary \ref{cor1}, the hypothesis $u^{r}(r,z) < \frac{1}{3r}+\frac{2r}{3}$ on $\p B_1 \cap \R^3_+$ in Theorem \ref{thm2} can be replaced by weakened to  $u^r(x) \le \alpha$ on $\p B_1 \cap \R^3_+$ with $\alpha < \frac{2\sqrt{2}}{3}$.
\end{remark}

The paper is organized as follows. In Section 2, we recall some basic definitions. Section 3 is devoted to the proof of Theorems \ref{thm1}--\ref{thm2} and Corollary \ref{cor1}.

\section{Preliminaries}

\subsection{Axisymmetric solutions}\label{Sect2.1}
In this subsection, we introduce the axially symmetric vector fields.
Let $(r,\theta, z)$ denote the cylindrical coordinates in $\mathbb{R}^3$, under which $(x_1, x_2, x_3)$ is represented as $(x_1, x_2, x_3) = (r \cos \theta, r\sin \theta, z)$.
The associated orthonormal basis ${\mathbf{e}_{r}, \mathbf{e}_{\theta}, \mathbf{e}_{z}}$ for this coordinate system is
\begin{align*}
\mathbf{e}_{r}=(\cos \theta, \sin \theta, 0), \quad \mathbf{e}_{\theta}=(-\sin \theta, \cos \theta,0), \quad \mathbf{e}_{z}=(0,0,1).
\end{align*}
In this coordinate system, a vector-valued function $\bw$ admits the representation
\begin{equation*}
    \bw = w^r(r,\te,z) \e_r + w^\te(r,\te,z) \e_\te + w^z(r,\te,z) \e_z, 
\end{equation*}
where $w^\theta$ is known as the swirl component. We say $\mathbf{w}$ is {\it axisymmetric} if it is of the form
\[
\bw = w^r(r,z) \e_r + w^\te(r,z) \e_\te + w^z(r,z) \e_z. 
\]


\section{Proof of the main results}

We first prove Theorems \ref{thm1}--\ref{thm2}. The key step is to prove $B^\theta \equiv 0$, since the MHD equations will reduce to the Navier--Stokes equations if $\B=0$, and the conclusion follows from the known results for the Navier--Stokes equations.

\begin{proof}[Proof of Theorem \ref{thm1}]
A direct computation shows
\begin{equation}\label{equationforB}
-\Delta B^{\te} + (\bu \cdot \nabla) B^{\theta} + \frac{1}{r^2} B^{\theta} - \frac{u^{r} B^{\theta}}{r} = 0.  
\end{equation}
For $0<\rho_1< \rho_2< \infty$, denote $\Omega_{1} = B_{\rho_{2}} \setminus \bar{B}_{\rho_1}$. Multiplying \eqref{equationforB} by $r B^{\theta}$ and integrating over $\Omega_{1}$, we obtain
\begin{align}\label{theorem1-1}
\int_{\Omega_{1}} (-\Delta B^{\te}) r B^{\te} 
+ r \bu \cdot \nabla (\frac12 |B^{\te}|^2)
+ \frac{1}{r} |B^{\te}|^{2} - u^{r} |B^{\te}|^{2} = 0.
\end{align}
For the first term in \eqref{theorem1-1}, we have
\begin{align*}
\int_{\Omega_{1}} -\Delta B^{\theta} (r B^{\theta})
= \int_{\Omega_{1}} r |\nabla B^{\te}|^{2} + \frac12\p_{r} |B^{\theta}|^2 
- \int_{\p \Omega_{1}} \frac{\p B^{\te}}{\p \boldsymbol \nu} r B^{\te}.
\end{align*}
Since $B^\theta$ is $(-1)$-homogeneous, i.\,e., $B^\te(\lambda x)=\lambda^{-1}B^\te(x)$, one finds that
\begin{align}\label{theorem1-2}
    \frac{x}{|x|} \cdot \nabla B^{\theta}(x) = -\frac{B^{\theta}(x)}{|x|}.
\end{align}
Then
\[
\int_{\p \Omega_{1}} \frac{\p B^{\te}}{\p \boldsymbol \nu} r B^{\te} = -\int_{\p B_{\rho_{2}}} \frac{r |B^{\te}|^{2}}{|x|} + \int_{\p  B_{\rho_{1}}} \frac{r |B^{\te}|^{2}}{|x|} = 0,
\]
owing to the self-similarity of $B^\theta$. Hence, 
\[
\int_{\Omega_{1}} -\Delta B^{\theta} (r B^{\theta}) = \int_{\Omega_{1}} r |\nabla B^{\theta}|^{2} + \frac{1}{2} \partial_{r} |B^{\theta}|^{2}.
\]
We claim that
\[
\int_{\Omega_{1}} \partial_{r} |B^{\theta}|^{2} = -\int_{\Omega_{1}} \frac{1}{r} |B^{\theta}|^{2}.
\]
To see this, direct calculation yields 
\begin{align*}
\int_{\Omega_{1}} \p_{r} |B^{\theta}|^{2}
=& \int_{\Omega_{1}} \div (|B^\te|^2 \e_r)-\frac{1}{r}|B^\te|^2\\
=& \int_{\partial \Omega_{1}} |B^{\theta}|^{2} \e_r \cdot \boldsymbol{\nu} - \int_{\Omega_{1}} \frac{1}{r} |B^{\theta}|^{2}\\
=&-\int_{\Omega_{1}} \frac{1}{r} |B^{\theta}|^{2},
\end{align*}
where the third equality is due to the self-similarity of $B^\te$.
Thus we have
\[
\int_{\Omega_{1}} -\Delta B^{\te} (r B^{\te}) = \int_{\Omega_{1}} r |\nabla B^{\te}|^{2} - \frac{1}{2r} |B^{\te}|^{2}.
\]
Integration by parts gives
\[
\int_{\Omega_{1}} r \bu \cdot \nabla (\frac12 |B^{\te}|^2)
=-\int_{\Omega_{1}} \div(r\bu) \frac12 |B^\te|^2 +\int_{\p \Omega_{1}}  \frac{r|B^\te|^2}{2}\bu \cdot \boldsymbol{\nu}
=-\int_{\Omega_{1}}  \frac{u^r}{2} |B^\te|^2,
\]
where we used $\div \bu = 0$ and the fact that both $B^{\theta}$ and $\mathbf{u}$ are $(-1)$-homogeneous.

Therefore, equation \eqref{theorem1-1} becomes
\[
\int_{\Omega_{1}} r|\nabla B^{\theta}|^{2} + |B^{\theta}|^{2}\left(\frac{1}{2r} - \frac{3}{2}u^{r}\right)=0.
\]
From \eqref{theorem1-2}, we obtain
\begin{align}\label{theorem1-3}
|\nabla B^{\theta}(x)| \geq \frac{1}{|x|}|B^{\theta}(x)|,
\end{align}
which gives
\begin{align*}
\int_{\Omega_{1}}|B^{\te}|^{2}\left(\frac{r}{r^{2}+z^{2}} + \frac{1}{2r} - \frac{3}{2}u^{r}\right)\le 0.
\end{align*}
Applying the coarea formula yields
\[
\log\big(\frac{\rho_2}{\rho_1}\big) \int_{\p B_1}|B^{\te}|^{2}\left(r + \frac{1}{2r} - \frac{3}{2}u^{r}\right)
=\int_{\Omega_{1}}|B^{\te}|^{2}\left(\frac{r}{r^{2}+z^{2}} + \frac{1}{2r} - \frac{3}{2}u^{r}\right)\le 0.
\]
Using the hypothesis $u^{r}(r,z) < \frac{1}{3r} + \frac{2r}{3}$ on $\p B_1$, we get
\[
\int_{\p B_1}|B^{\te}|^{2}\left(r + \frac{1}{2r} - \frac{3}{2}u^{r}\right) = 0,
\]
which implies $B^\te=0$ on $\p B_1$. Combining with the fact that $B^\te$ is self-similar, we have $B^\te \equiv0$.
Thus $\mathbf{B} \equiv 0$ and equations \eqref{MHD} become the Navier--Stokes equations. Since any axisymmetric self-similar solution to the stationary Navier--Stokes equations in $\mathbb{R}^3\setminus \{0\}$ is a Landau solution (see, for example, \cite[Theorem 8.1]{Tsai18}), the proof of Theorem \ref{thm1} is complete.
\end{proof}

The proof of Theorem \ref{thm2} is in the same spirit as that of Theorem \ref{thm1}. 

\begin{proof}[Proof of Theorem \ref{thm2}] 
During the proof, for $0<\rho_1<\rho_2 < \infty$, denote $\Omega_{2}= (B_{\rho_2}\setminus \bar B_{\rho_1}) \cap \mathbb R^3_+$.
Multiplying  \eqref{equationforB} by $rB^\theta$ and integrating over $\Omega_{2}$, we obtain the same identity \eqref{theorem1-1}, with $\Omega_1$ replaced by $\Omega_2$.
 
Integration by parts yields
\begin{align}\label{theorem2-1}
\int_{\Omega_{2}} r|\nb B^{\te}|^{2} + |B^{\te}|^2 \left(\frac{1}{2r} - \frac{3u^{r}}{2} \right)
+\int_{\p \Omega_{2}} -\frac{\p B^\te}{\p \boldsymbol{\nu}} r B^\te
+\frac{|B^\te|^2}{2} ( \e_r \cdot \boldsymbol{\nu} +r \bu \cdot \boldsymbol{\nu})=0.
\end{align}
By the self-similarity of $B^\theta$, the boundary integral simplifies to
\begin{align*}
\int_{\p \Omega_{2}}  -\frac{\p B^{\te}}{\p \boldsymbol{\nu} } r B^{\te} 
+\frac{|B^\te|^2}{2} (\e_r \cdot \boldsymbol{\nu}+r \bu \cdot \boldsymbol{\nu})
=\int_{S}  r B^{\te}\p_z B^\te  +\frac{|B^\te|^2}{2} ( \e_r \cdot \boldsymbol{\nu}+ r \bu \cdot \boldsymbol{\nu}),
\end{align*}
where $S=\{ x \in \p \R^3_+: \r_1 \le |x| \le \r_2 \}$ and $\boldsymbol{\nu}=-\e_z=(0,0,-1)$.
We claim that 
\[
\int_{S}  r B^{\te}\p_z B^\te  +\frac{|B^\te|^2}{2} ( \e_r \cdot \boldsymbol{\nu} + r \bu \cdot \boldsymbol{\nu})=0
\]
under either the no-slip or the Navier slip boundary condition.
First, a direct calculation gives
\[
\curl \mathbf{B} = -\partial_z B^\theta(r,z) \mathbf{e}_r + \frac{1}{r} \partial_r (r B^{\theta}) \mathbf{e}_z.
\]
Since $\curl \mathbf{B} \times \mathbf{e}_z = 0$ on $\partial \mathbb{R}^3_+$ under both boundary conditions, we obtain
\begin{align}\label{theorem2-2}
    \partial_z B^\theta = 0 \quad \text{on } \partial \mathbb{R}^3_+.
\end{align}
For the remaining terms, under the no-slip condition \eqref{noslipbc}, we have $\mathbf{u} = 0$ on $\partial \mathbb{R}^3_+$; under the Navier slip condition \eqref{navierslipbc}, we have $\mathbf{u} \cdot \mathbf{e}_z = 0$.
In both cases, since $\mathbf{e}_r$ lies in the tangent plane of $\partial \mathbb{R}^3_+$, it follows that
\begin{align}\label{theorem2-3}
\int_{S} \frac{|B^\theta|^2}{2} ( \e_r \cdot \boldsymbol{\nu} + r \mathbf{u} \cdot \boldsymbol{\nu}) = 0.
\end{align}
From \eqref{theorem2-2} and \eqref{theorem2-3}, the claim is verified.
Thus \eqref{theorem2-1} becomes
\[
\int_{\Omega_{2}} r|\nb B^{\te}|^{2} + |B^{\te}|^2 \left(\frac{1}{2r} - \frac{3u^{r}}{2} \right)=0.
\]

The remainder of the proof proceeds as in Theorem \ref{thm1}.
Using \eqref{theorem1-3}, the upper bound on $u^r$, and the self-similarity of $B^\theta$ and $u^r$, we conclude that $\mathbf{B} \equiv 0$. The vanishing of $\mathbf{u}$ is then a direct consequence of \cite[Theorem 5.1]{KMT18}, which completes the proof.
\end{proof}

Next, we give the proof of Corollary \ref{cor1}.

\begin{proof}[Proof of Corollary \ref{cor1}]
Note that the minimum of the function $\frac{1}{3r} + \frac{2r}{3}$ over $\p B_1$ is $\frac{2\sqrt{2}}{3}$. 
Consequently, the conclusion of Theorem \ref{thm1} applies, which completes the proof.
\end{proof}

\section*{Acknowledgements}
The author would like to thank Professors Yun Wang and Nicola De Nitti for many helpful discussions.
The work of the author is partially supported by the Postgraduate Research \& Practice Innovation Program of Jiangsu Province via grant KYCX24\_3285.
The author also appreciates the reviewers for their careful reading and helpful comments.

\bibliographystyle{alpha}
\bibliography{ref}

\end{document}